\documentclass{amsart}

\usepackage{hyperref}
\usepackage{amsmath}

\usepackage{amssymb}
\usepackage{mathrsfs}
\usepackage{amscd}
\usepackage{graphicx}
\usepackage{mathdots}
\usepackage{gensymb}
\usepackage{epstopdf}
 \usepackage{todonotes}
 \usepackage{youngtab}
 \usepackage{wrapfig}
 \usepackage{bbm}
 \usepackage{bm}
 \usepackage{verbatim}
 \usepackage[margin=1.1in]{geometry}

 \usepackage[nocompress]{cite}

\usepackage{epsfig}       
\usepackage{epic,eepic}        

\newtheorem{thm}{Theorem}[section]
\newtheorem{prop}[thm]{Proposition}

\newtheorem{cor}[thm]{Corollary}

\theoremstyle{definition}

\theoremstyle{remark}
\newtheorem{remark}[thm]{Remark}

\numberwithin{equation}{section}

\newcommand{\Z}{\mathbb{Z}}

\newcommand{\C}{\mathbb{C}}
\newcommand{\R}{\mathbb{R}}

\newcommand{\ttP}{\mathtt{P}}
\newcommand{\ttE}{\mathtt{E}}

\begin{document}

\title[The Opdam--Cherednik kernel is the Laplace transform of a positive measure]{The Opdam--Cherednik kernel is the Laplace transform of a positive measure}
 
 \author{Colin McSwiggen}
\address{Institute of Mathematics, Academia Sinica}
\email{csm@as.edu.tw}

\author{Siddhartha Sahi}
\address{Department of Mathematics, Rutgers University}
\email{sahi@math.rutgers.edu}

\begin{abstract}
We prove that the Opdam--Cherednik kernel, also known as the nonsymmetric Opdam hypergeometric function, can be written as the Laplace transform of a positive measure supported on the convex hull of the Weyl group orbit of its argument.  As a consequence, the trigonometric Dunkl intertwining operator is positivity preserving.  The main ingredient in the proof is a new formula for the Opdam--Cherednik kernel as a degeneration of nonsymmetric Macdonald polynomials.  As a further application, we prove majorization inequalities for Macdonald polynomials and Heckman--Opdam hypergeometric functions associated with arbitrary root systems.
\end{abstract}

\maketitle

 \tableofcontents

\section{Introduction}

Special functions are like people: the same traits that make them interesting often also make them difficult.  A Bessel or hypergeometric function has the personality of a coy socialite.  You run into them everywhere, but it takes skill and tenacity to get them to reveal anything about themselves. On one hand, they appear so frequently in so many disparate subjects that it is hard not to regard them as a kind of basic component of reality.  On the other, they are evasive objects, impossible to write down in closed form, admitting only indirect or approximate descriptions via differential equations, series expansions, and the like.  This, then, is both the challenge and the allure of special functions: they are equal parts fundamental and opaque.

To work with such objects, we need to relate them to something more tangible.  One of the oldest and most useful techniques is to derive an integral formula: one proves that the function one would like to study can be written as the integral of something else that is easier to analyze.  A foundational example is the Poisson integral representation of the Bessel function of the first kind:
\begin{equation} \label{eqn:poisson-bessel}
J_\alpha(x) = \frac{1}{\sqrt{\pi}\,\Gamma\!\left(\alpha + \frac{1}{2}\right)} \left(\frac{x}{2}\right)^{\alpha} \int_{-1}^{1} (1-t^2)^{\alpha - \frac{1}{2}} \cos(xt)\, dt, \qquad \mathrm{Re}(\alpha) > -\frac{1}{2},
\end{equation}
which says that, up to normalization, $(x/2)^{-\alpha} J_\alpha(x)$ is the cosine transform of a measure with positive density $(1-t^2)^{\alpha - 1/2}$ on $[-1,1]$.  From (\ref{eqn:poisson-bessel}), one easily obtains global bounds, asymptotics at infinity, and many other properties of $J_\alpha$.

The main contribution of this paper is a far-reaching generalization of (\ref{eqn:poisson-bessel}), which has been speculated to hold for roughly two decades and has remained a notable unsolved problem in Dunkl theory.  It gives an integral representation for the Opdam--Cherednik kernel, also known as the nonsymmetric Opdam hypergeometric function, of the form:
\begin{equation} \label{eqn:GLP1}
G_{k,s-\rho}(x) = \int_V e^{\langle s, y \rangle} \, \pi_{k, x}(dy), \qquad k \ge 0,
\end{equation}
where $G_{k,s}$ is the Opdam--Cherednik kernel for a root system $\Phi \subset V \cong \R^n$, $k$ is a multiplicity parameter, $\rho$ is the $k$-weighted Weyl vector, and $\pi_{k,x}$ is a probability measure supported on the convex hull of the Weyl group orbit of $x$. In words, $G_{k,s-\rho}$ is the Laplace transform of $\pi_{k, x}$.  In rank 1, symmetrizing over the Weyl group and passing to the rational limit recovers (\ref{eqn:poisson-bessel}).\footnote{In fact this procedure only directly recovers (\ref{eqn:poisson-bessel}) for real $\alpha = k - \tfrac{1}{2}$, but that implies the full identity by analytic continuation.}

The formula (\ref{eqn:GLP1}) can also be regarded as an extension of two further classical results: Harish-Chandra's integral formula for spherical functions on Riemannian symmetric spaces \cite{HC2} and Kostant's convexity theorem for the Iwasawa decomposition \cite{kostant1973}.  The symmetrization of $G_{k,s}$ over the Weyl group is the Heckman--Opdam hypergeometric function $F_{k,s}$, which for special multiplicity parameters $k$ recovers the spherical functions on Riemannian symmetric spaces of noncompact type.  Harish-Chandra's formula expresses a spherical function on a
symmetric space $G/K$ as an integral of an exponential over the compact
group $K$:
\begin{equation} \label{eqn:HC}
\varphi_\lambda(a) = \int_K e^{\langle \lambda - \rho,\, H(au) \rangle}\, du,
\end{equation}
where $H : G \to \mathfrak{a}$ is the Iwasawa projection associated with a
decomposition $G = K (\exp \mathfrak a) N$, and $du$ is the normalized Haar
measure on $K$. Kostant's convexity theorem identifies the image of this
projection: as $u$ ranges over $K$, the points $H(au)$ fill out exactly the
convex hull of the Weyl group orbit of $\log a$. Pushing the Haar measure
forward along $H(a\,\cdot\,)$ therefore turns (\ref{eqn:HC}) into an integral
of $e^{\langle \lambda - \rho,\, y\rangle}$ against a probability measure on this convex hull, and for the appropriate choice of $k$ corresponding to the symmetric space, this pushforward measure is precisely $e^{\langle \rho, y \rangle}$ times the symmetrization of $\pi_{k, \log a}$.

The Laplace transform representation (\ref{eqn:GLP1}) is thus both a nonsymmetric refinement and a continuous deformation of (\ref{eqn:HC}), providing an integral formula that remains valid even at a level of generality where the functions do not arise from a symmetric space and there is no longer a group $K$ to integrate over.  In all cases, regardless of whether the multiplicity parameter $k$ corresponds to a symmetric space, the support of the integration measure is contained in the same polytope identified by the Kostant convexity theorem.  In symplectic geometry, well-known generalizations of the Kostant convexity theorem due to Atiyah \cite{Atiyah1982}, Guillemin and Sternberg \cite{GuilleminSternberg1982}, and Kirwan \cite{Kirwan1984} show that Duistermaat--Heckman measures for compact Hamiltonian manifolds, which are obtained from the pushforward of the Liouville measure by a moment map \cite{DH}, are supported on convex polytopes.  The measure $\pi_{k,x}$ in (\ref{eqn:GLP1}) can be viewed as an analogue of a Duistermaat--Heckman measure that exists even without a Hamiltonian manifold or a Lie group action in the background.

The main technical ingredient in the proof of (\ref{eqn:GLP1}) is a new formula for the Opdam--Cherednik kernel as a degeneration of nonsymmetric Macdonald polynomials, which is of independent interest.  Using this result, the measure $\pi_{k,x}$ is constructed as a limit of discretely supported measures that can be computed explicitly via the Ram--Yip formula for the Macdonald polynomial coefficients \cite{RamYip2011}.  This construction opens the possibility of a more detailed study of the measures $\pi_{k,x}$, which we hope will enable progress on a number of analytic questions in Dunkl theory, as we discuss below in Section \ref{sec:appl-analysis}.

A consequence of the integral formula (\ref{eqn:GLP1}) is that the trigonometric Dunkl intertwining operator $\mathcal{V}_k$ is positivity preserving for all $k \ge 0$, as shown in Corollary \ref{cor:positivity} below. The analogue of this statement in the rational setting is a celebrated result of R\"osler \cite{RoslerPositivity}, who proved a similar Laplace transform formula for the rational Dunkl kernel $E_{k,s}$. The Dunkl kernel can be obtained from the Opdam--Cherednik kernel via the rational limit, and accordingly (\ref{eqn:GLP1}) can be seen as a generalization of R\"osler's formula that recovers it as a limiting case; we make this precise in Remark \ref{rem:rational}.  Since our methods are distinct from those of \cite{RoslerPositivity}, the results in this paper also offer new proofs of the corresponding results in the rational theory.

The trigonometric analogue of R\"osler's theorem has been sought for twenty years or so. The corresponding intertwining operator $\mathcal{V}_k$, the existence of which had been a matter of speculation, was constructed rigorously in 2010 by Trim\`eche \cite{Trimeche2010}, who also showed that the Opdam--Cherednik kernel could be written as the Laplace transform of a \emph{distribution} supported on the convex hull of the Weyl orbit of its argument. This distribution is our $\pi_{k,x}$; the problem that remained after \cite{Trimeche2010} was to prove that $\pi_{k,x}$ is in fact a positive measure.  Even in rank 1, this positivity statement has resisted a number of prior attempts at both proof \emph{and} disproof, as outlined in \cite[Remarks 3.3 and 3.4]{AnkerPositivity}.  The rank-1 case was finally settled by Anker \cite{AnkerPositivity}, who derived a closed expression for the representing kernel and read off its positivity directly. In higher rank, however, the positivity of $\mathcal{V}_k$ has remained open. The present paper resolves this question.

We hope that the solution of this problem will open new avenues for progress in trigonometric Dunkl theory, just as R\"osler's theorem has done for the rational theory.  In Section \ref{sec:majorization} we develop one such application, establishing log-convexity in the spectral parameter for $G_{k,s}$ and $F_{k,s}$, together with Muirhead-type majorization inequalities for $F_{k,s}$.  Such inequalities have been a focus of much recent work on symmetric polynomials and related special functions; the inequalities obtained here extend this line of results to Heckman--Opdam hypergeometric functions for arbitrary root systems, resolving a conjecture of McSwiggen and Novak \cite{MN-majorization} and generalizing the type-$A$ inequalities of \cite{McSwiggenSahiMajorization}.  We also prove related results on Macdonald polynomials.

\subsection{Overview}

Section \ref{sec:prelim} collects background on the Opdam--Cherednik kernel and on Macdonald polynomials, and fixes the correspondence between the data used as input in their respective constructions.

Section \ref{sec:main} contains our main results.  The central theorem of the paper is Theorem \ref{thm:GLP}, the Laplace transform representation (\ref{eqn:GLP1}). Its proof rests on Theorem \ref{thm:mac-hgf}, which expresses the Opdam--Cherednik kernel as a limit of normalized nonsymmetric Macdonald polynomials. Corollary \ref{cor:positivity} records the immediate consequence that the trigonometric Dunkl intertwining operator $\mathcal{V}_k$ is positivity preserving.

Section \ref{sec:majorization} develops applications to spectral log-convexity and majorization inequalities. Propositions \ref{prop:E-lc} and \ref{prop:conv-mac} establish log-convexity in the spectral parameter, and in the symmetric case $W_0$-convexity, for the nonsymmetric and symmetric Macdonald polynomials evaluated at certain lattice points. Theorem \ref{thm:maj-mac} then characterizes the $W_0$-majorization order through inequalities between symmetric Macdonald polynomials. Corollaries \ref{cor:G-lc} and \ref{cor:conv-HGF} give log-convexity of $G_{k,s}$ and $F_{k,s}$ in the spectral parameter for all real arguments, and Theorem \ref{thm:HGF-maj} deduces the corresponding majorization inequalities for the Heckman--Opdam hypergeometric function, resolving a conjecture of McSwiggen and Novak \cite{MN-majorization}.

Section \ref{sec:appl-analysis} illustrates further possible applications in analysis via a single demonstrative example. We give an alternative integral formula for the Heckman--Opdam heat kernel (Proposition \ref{prop:heat-rep}), and we explain how sharp two-sided estimates on the heat kernel would follow from regularity results on the measures $\pi_{k,x}$.

\section{Preliminaries}
\label{sec:prelim}

\subsection{The Opdam--Cherednik kernel}
\label{sec:prelim-OC}

Let $V$ be a Euclidean space with inner product $\langle \cdot, \cdot \rangle$, $\Phi$ an irreducible root system spanning $V$, $\Phi^+$ a choice of positive roots, $P$ the weight lattice, $P_+$ the cone of dominant weights, and $W_0$ the Weyl group generated by the reflections $s_\alpha$ for $\alpha \in \Phi^+$, where
\[
s_\alpha x = x - 2 \frac{\langle \alpha, x \rangle}{\langle \alpha, \alpha \rangle} \alpha, \qquad x \in V.
\]
Fix a $W_0$-invariant multiplicity parameter $k : \Phi \to [0,\infty)$ and set
\begin{equation} \label{eqn:HO-rho}
\rho = \frac{1}{2} \sum_{\alpha \in \Phi^+} k(\alpha) \, \alpha.
\end{equation}
For $\xi \in V$, the {\it Dunkl--Cherednik operator} $D_{k,\xi}$ is the differential-reflection operator defined by
\begin{equation} \label{eqn:cherednik-op-def}
D_{k,\xi} f(x) = \partial_\xi f(x) + \sum_{\alpha \in \Phi^+} k(\alpha) \frac{\langle \alpha, \xi \rangle}{1-e^{-\langle \alpha, x \rangle}} \big[ f(x) - f(s_\alpha x) \big] - \langle \rho, \xi \rangle f(x), \qquad f \in C^1(V).
\end{equation}

For each $s$ in the complexification $V_\C$, there is a unique function $G_{k,s} \in C^\infty(V)$, called the \emph{Opdam--Cherednik kernel}, satisfying the system of eigenvalue equations
\begin{equation} \label{eqn:hyperbolic-eigproblem}
D_{k, \xi} G_{k, s} = \langle s, \xi \rangle G_{k, s} \qquad \text{for all } \xi \in V,
\end{equation}
and normalized such that $G_{k,s}(0) = 1$; see \cite{Op}. For each fixed $x$, $G_{k,s}(x)$ is a holomorphic function of $s \in V_\C$; the point $s$ is referred to as the \emph{spectral parameter}.

The {\it Heckman--Opdam hypergeometric function} $F_{k, s}$ is the symmetrization of $G_{k,s}$ over $W_0$:
\begin{equation} \label{eqn:HGF-def}
F_{k, s}(x) = \frac{1}{|W_0|} \sum_{w \in W_0} G_{k, s}(w(x)).
\end{equation}
It is a multivariable generalization of the classical Gauss hypergeometric function and was originally constructed and studied in the series of papers \cite{RSHF1, RSHF2, RSHF3, RSHF4}.

The \emph{trigonometric Dunkl intertwining operator} $\mathcal{V}_k$ is the unique continuous linear operator on $C^\infty(V)$ satisfying
\begin{equation}
    D_{k,\xi} \circ \mathcal{V}_k = \mathcal{V}_k \circ \partial_\xi \qquad \text{for all } \xi \in V,
\end{equation}
and normalized by $\mathcal{V}_k(f)(0) = f(0)$; see \cite{Trimeche2010} for existence and uniqueness. The Opdam--Cherednik kernel can be expressed in terms of $\mathcal{V}_k$ via
\begin{equation} \label{eqn:G-V-exp}
    G_{k, s}(x) = \mathcal{V}_k\bigl(e^{\langle s, \,\cdot\,\rangle}\bigr)(x),
\end{equation}
and consequently the hypergeometric function satisfies
\begin{equation}
    F_{k, s}(x) = \frac{1}{|W_0|}\sum_{w \in W_0} \mathcal{V}_k\bigl(e^{\langle w(s),\,\cdot\,\rangle}\bigr)(x).
\end{equation}

For $\eta \in P$, define
\begin{equation} \label{eqn:bardef}
\bar{\eta} \;=\; \eta - w_\eta^{-1}(\rho),
\end{equation}
where $w_\eta \in W_0$ is the shortest element such that $w_\eta(\eta)$ is antidominant. In particular, for strictly dominant $\lambda$ we have $\bar\lambda = \lambda + \rho$. When the spectral parameter equals $\bar\eta$, the functions $F_{k,s}$ and $G_{k,s}$ specialize to exponential polynomials known as the symmetric and nonsymmetric \emph{Heckman--Opdam Jacobi polynomials} \cite{Op, HS},
\begin{equation} \label{eqn:jac-hgf}
\ttP^{(k)}_\lambda(x) \;:=\; F_{k,\,\bar\lambda}(x), \quad \lambda \in P_+, \qquad \quad
\ttE^{(k)}_\eta(x) \;:=\; G_{k,\,\bar\eta}(x), \quad \eta \in P,
\end{equation}
both normalized so that $\ttP^{(k)}_\lambda(0) = \ttE^{(k)}_\eta(0) = 1$, since $F_{k,s}(0) = G_{k,s}(0) = 1$.

  For further details of the constructions above, we refer the reader to the reviews \cite{AnkerDunklNotes, HS, OpdamLectureNotes}.

\subsection{Macdonald polynomials} \label{sec:prelim-mac}

In \cite{MacdonaldAffine}, Macdonald gave a construction of orthogonal polynomials associated to certain pairs $(S, S')$ of irreducible affine root systems with common finite Weyl group $W_0$.  These generalize the more familiar Macdonald polynomials of type $A$.  In the construction, each pair $(S,S')$ is associated with a pair of finite root systems $(R, R')$ and lattices $(L, L') \subset V$, where $V$ is a real Euclidean space. The three cases considered by Macdonald \cite[(1.4.1)--(1.4.3)]{MacdonaldAffine} are:
\begin{enumerate}
\item[(1.4.1)] $S = S(R)$, $S' = S(R^\vee)$, $L = P$, $L' = P^\vee$, where $R$ is any reduced irreducible root system, $P$ and $P^\vee$ are its weight and coweight lattices, and $R^\vee$ is the dual root system.
\item[(1.4.2)] $S = S' = S(R)^\vee$, $L = L' = P^\vee$, with $R$ non-simply laced.
\item[(1.4.3)] $S = S'$ is of type $(C_n^\vee, C_n)$, $L = L' = \mathbb{Z}^n$, corresponding to the non-reduced root system $BC_n$.
\end{enumerate}
Here $S(R) = \{ \alpha + n : \alpha \in R, \, n \in \Z \}$ indicates the affine root system associated to $R$, and $R^\vee$ indicates the dual root system. We always assume that root systems are crystallographic. In case \cite[(1.4.3)]{MacdonaldAffine} above, the Macdonald polynomials coincide with the Koornwinder polynomials \cite{KoornwinderBC, SahiNonsymmetric1999}.

Given a pair $(S,S')$ according to one of the cases above, let $q \in (0,1)$ and let $t_a \in (0,1)$ for each $a \in S$, with $t_{w(a)} = t_a$ for all $w$ in the affine Weyl group $W$ of $S$. These parameters determine a $W$-invariant label function $k$ on $S$ via \cite[(5.1.1)]{MacdonaldAffine}:
\[
q^{k(a)} = t_a t_{2a}, \qquad q^{k(2a)} = t_{2a},
\]
where $t_{2a} = 1$ if $2a \notin S$. In the reduced cases \cite[(1.4.1)--(1.4.2)]{MacdonaldAffine} there is a single parameter $t_i = t_{a_i}$ for each orbit of simple roots, while in the non-reduced case \cite[(1.4.3)]{MacdonaldAffine} of type $(C_n^\vee, C_n)$ there are additionally parameters $t_{2a_i}$ for the relevant roots.

Following the constructions of \cite{MacdonaldAffine}, the above data determine a family of Macdonald polynomials, together with an additional family of \emph{dual} Macdonald polynomials.  Throughout this paper we work in the \emph{self-dual} case, in which $(S,L) = (S',L')$ and the Macdonald polynomials coincide with their duals. This holds in Macdonald's case \cite[(1.4.1)]{MacdonaldAffine} when $R$ is simply laced, in case \cite[(1.4.2)]{MacdonaldAffine} when $R$ is non-simply laced, and in case \cite[(1.4.3)]{MacdonaldAffine} when restricted to a distinguished $3$-parameter subfamily of the Koornwinder polynomials \cite{KoornwinderBC, SahiNonsymmetric1999}. In Sahi's parameterization \cite{SahiNonsymmetric1999} this subfamily is specified by $t_0 = u_n$ and $u_0 = 1$; the condition $t_0 = u_n$ ensures self-duality, while the condition $u_0 = 1$ ensures that label functions on $S$ are in bijection with the $3$-valued multiplicity functions on $BC_n$ via $t_a = q^{k(a)}$.  Having restricted our attention to the self-dual case, we write $P$ for the lattice $L = L'$, and
\[
P_+ = \{\lambda \in P : \langle \lambda, \alpha_i \rangle \geq 0 \text{ for all simple roots } \alpha_i \text{ of } R\}
\]
for the (weakly) dominant cone ($L_{++}$ in the notation of \cite{MacdonaldAffine}). These notational choices serve to clarify the relationship between the Macdonald and Jacobi polynomials, as we explain in more detail below.

Let $K[P]$ be the group algebra of the lattice $P$ with coefficients in the field
\[
K = \mathbb{Q}(q^{1/e},\, t_a^{1/2},\, t_{2a}^{1/2} : a \in S) \; \subset \; \mathbb{R},
\]
where $e$ is a positive integer depending only on the root system, as defined in \cite[(1.4.5)]{MacdonaldAffine}. Write $K[P]^{W_0}$ for the algebra of $W_0$-invariant elements of $K[P]$.

For $\lambda \in P_+$, the \emph{symmetric Macdonald polynomial} $P^{(q,k)}_\lambda \in K[P]^{W_0}$ is the unique element satisfying \cite[(5.3.1)]{MacdonaldAffine}:
\begin{enumerate}
\item[(i)] $P^{(q,k)}_\lambda = m_\lambda + \text{lower terms}$, where $m_\lambda = \sum_{\mu \in W_0\lambda} e^\mu$ is the $W_0$-orbit-sum and ``lower terms'' means a $K$-linear combination of $m_\mu$ with $\mu \in P_+$, $\mu < \lambda$;
\item[(ii)] $\langle P^{(q,k)}_\lambda, m_\mu \rangle = 0$ for all $\mu \in P_+$ with $\mu < \lambda$,
\end{enumerate}
where $\langle\cdot,\cdot\rangle$ is the scalar product on $K[P]^{W_0}$ defined in \cite[(5.1.29)]{MacdonaldAffine} and $<$ is the dominance order on $P_+$\cite[(2.6.3)--(2.6.4)]{MacdonaldAffine}.

The \emph{nonsymmetric Macdonald polynomial} $E^{(q,k)}_\eta \in K[P]$ for $\eta \in P$ is the unique element satisfying \cite[(5.2.1)]{MacdonaldAffine}:
\begin{enumerate}
\item[(i)] $E^{(q,k)}_\lambda = e^\lambda + \text{lower terms}$, where ``lower terms'' means a $K$-linear combination of $e^\mu$ with $\mu \in P,$ $\mu < \lambda$,
\item[(ii)] $(E^{(q,k)}_\lambda, e^\mu) = 0$ for all $\mu \in P$ with $\mu < \lambda$,
\end{enumerate}
where $(\cdot,\cdot)$ is the scalar product on $K[P]$ defined in \cite[(5.1.17)]{MacdonaldAffine} and $<$ now indicates the extension of the dominance order from $P_+$ to $P$ as described in \cite[\S2.7]{MacdonaldAffine}.  The weights $\lambda$ and $\eta$ are the \emph{spectral parameters}.

Define the $k$-deformed Weyl vector
\begin{equation} \label{eqn:mac-rho}
\rho = \frac{1}{2} \sum_{\alpha \in R^+} k(\alpha^\vee)\, \alpha,
\end{equation}
where $R^+$ is the set of positive roots of $R$, and $\alpha^\vee = \tfrac{2}{\langle \alpha, \alpha \rangle} \alpha$ is the coroot associated to $\alpha$.  This definition of $\rho$ differs from (\ref{eqn:HO-rho}), but under the correspondence with the Heckman--Opdam theory described below, the two definitions are equal.

Elements of $K[P]$ are evaluated at a point $x \in V$ via $e^\nu \mapsto q^{\langle \nu, x \rangle}$. Since $P^{(q,k)}_\lambda(-\rho) \neq 0$ and $E^{(q,k)}_\eta(-\rho) \neq 0$ \cite[\S5.2--3]{MacdonaldAffine}, \cite{SahiNonsymmetric1999}, we define the \emph{normalized symmetric} and \emph{normalized nonsymmetric Macdonald polynomials}
\[
\ttP^{(q,k)}_\lambda \;:=\; \frac{P_\lambda}{P_\lambda(-\rho)}, \quad \lambda \in P_+, \qquad \quad
\ttE^{(q,k)}_\eta \;:=\; \frac{E_\eta}{E_\eta(-\rho)}, \quad \eta \in P,
\]
so that $\ttP^{(q,k)}_\lambda(-\rho) = \ttE^{(q,k)}_\eta(-\rho) = 1$.

Three fundamental properties of the Macdonald polynomials are crucial to the arguments that follow. First, they satisfy duality identities \cite[(5.2.6) and (5.3.6), see also p.~23 and p.~35 for notation]{MacdonaldAffine}, \cite{SahiNonsymmetric1999}:
\begin{align}
\ttP^{(q,k)}_\lambda(\bar \mu) &\;=\; \ttP^{(q,k)}_\mu(\bar \lambda) \qquad \text{for all } \lambda, \mu \in P_+, \label{eqn:mac-sym-dual-new} \\
\ttE^{(q,k)}_\eta(\bar\mu) &\;=\; \ttE^{(q,k)}_\mu(\bar\eta) \qquad \text{for all } \eta, \mu \in P, \label{eqn:mac-nonsym-dual-new}
\end{align}
where $\bar\eta = \eta - w_\eta^{-1}(\rho)$ as defined above in (\ref{eqn:bardef}). Second, by the Ram--Yip formula \cite[Theorem~3.1]{RamYip2011}, the symmetric and nonsymmetric Macdonald polynomials admit positive expansions of the following forms:
\begin{alignat}{3} \label{eqn:mac-mon-pos-sym}
\ttP^{(q,k)}_\lambda &&= b_{\lambda\lambda}\, m_\lambda &+ \sum_{\nu < \lambda} b_{\lambda\nu}\, m_\nu, &\qquad b_{\lambda\nu} \geq 0,\quad b_{\lambda\lambda} > 0, \\
\label{eqn:mac-mon-pos-ns}
\ttE^{(q,k)}_\eta &&= c_{\eta\eta}\, e^\eta &+ \sum_{\nu < \eta} c_{\eta\nu}\, e^\nu, &\qquad c_{\eta\nu} \geq 0,\quad c_{\eta\eta} > 0,
\end{alignat}
where $<$ represents the dominance order on $P_+$ in (\ref{eqn:mac-mon-pos-sym}) and its extension to $P$ (see \cite[\S2.7]{MacdonaldAffine}) in (\ref{eqn:mac-mon-pos-ns}).

Finally, the Macdonald polynomials recover the Heckman--Opdam Jacobi polynomials in the limit $q \to 1$. For any root system $\Phi$ and positive multiplicity parameter $k$, there is an affine root system $S$ in one of the self-dual cases above with the same finite Weyl group $W_0$, such that $L=L'=P$ is the weight lattice of $\Phi$, together with a positive label function on $S$ also denoted $k$, such that
\begin{equation} \label{eqn:mac-jac-lim-new}
\ttP^{(k)}_\lambda(x) = \lim_{\substack{q \to 1 \\ t_a = q^{k(a)}}} \ttP^{(q,k)}_\lambda\!\bigl((\log q)^{-1} x\bigr), \qquad
\ttE^{(k)}_\eta(x) = \lim_{\substack{q \to 1 \\ t_a = q^{k(a)}}} \ttE^{(q,k)}_\eta\!\bigl((\log q)^{-1} x\bigr),
\end{equation}
uniformly for $x$ in compact subsets of $V$. For the reduced cases this follows from results of \cite{RSHF1, RSHF2, RSHF3, RSHF4} and \cite{Op} by considering the $q \to 1$ limits of the inner products $\langle \cdot, \cdot \rangle$ and $(\cdot, \cdot)$; for the Koornwinder case see \cite[\S6]{SahiNonsymmetric1999}.

To make the correspondence concrete, the finite root system $R$ associated to $S$ in Macdonald's setup is related to the root system $\Phi$ as follows, where $P(R)$ indicates the weight lattice of $R$, $P^\vee(R)$ indicates the coweight lattice of $R$, and $P(\Phi) = P$ is the weight lattice of $\Phi$:
\begin{itemize}
    \item Reduced, simply-laced (case \cite[(1.4.1)]{MacdonaldAffine}): $\Phi = R$, and $L = L' = P(R) = P(\Phi)$.
    \item Reduced, non-simply-laced (case \cite[(1.4.2)]{MacdonaldAffine}): $\Phi = R^\vee$, and $L = L' = P^\vee(R) = P(\Phi)$.
    \item Type $BC_n$ (Koornwinder, case \cite[(1.4.3)]{MacdonaldAffine}): $L = L' = \mathbb{Z}^n = P(\Phi)$, where $\Phi = BC_n$.
\end{itemize}
In all three cases, the notation $P$ for the common lattice $L = L'$ is justified by the fact that this is the weight lattice of $\Phi$, which allows the Macdonald polynomials indexed by $\eta \in L$ to recover the Heckman--Opdam Jacobi polynomials indexed by $\eta \in P(\Phi)$ in the limit $q \to 1$. Moreover, in these cases the definitions (\ref{eqn:HO-rho}) and (\ref{eqn:mac-rho}) of $\rho$ in both theories yield the same vector, so the notation $\bar \eta$ is unambiguous. Note that $P$ does \emph{not} coincide with the weight lattice of $R$ unless $\Phi$ is simply laced.

\section{Main results}
\label{sec:main}

Throughout this section we fix an irreducible root system $\Phi$ spanning a Euclidean space $V \cong \R^n$.  Given any nonnegative multiplicity parameter $k$ on $\Phi$, we also fix an affine root system $S=S'$ in one of the self-dual cases of Section \ref{sec:prelim-mac}, together with a label function on $S$, also denoted $k$, such that the limit (\ref{eqn:mac-jac-lim-new}) holds.

\begin{thm} \label{thm:mac-hgf}
     For $x \in V$, let $\mu(x) = \lfloor (\log q)^{-1} \cdot x \rfloor \in P,$
    where the floor function is applied coordinate-wise in the fundamental weight basis. Then for all $s \in V_\C$ and strictly positive $k$,
    \begin{equation} \label{eqn:mac-hgf}
        G_{k,s}(x) = \lim_{q \to 1} \ttE^{(q,k)}_{\mu(x)}(s),
    \end{equation}
    and the convergence is uniform for $(s,x)$ in compact subsets of $V_\C \times V$.
\end{thm}

\begin{proof}
    The claim follows by a uniform estimate at the points $s = \bar\eta$, $\eta \in P$, together with an iterated application of Carlson's theorem, which we use in the following form: an entire function $f$ of a single complex variable that is bounded on the imaginary axis, satisfies $|f(z)| \le C e^{\tau|z|}$ for some $C, \tau > 0$, and vanishes at every sufficiently large positive integer, vanishes identically.

    We will apply this to the difference of both sides of (\ref{eqn:mac-hgf}), one coordinate at a time.  Accordingly, we first show that both $G_{k,s}(x)$ and $\ttE^{(q,k)}_{\mu(x)}(s)$ satisfy a growth bound of the form
    \begin{equation} \label{eqn:carlson-bd}
        |f(s)| \;\le\; C e^{\tau |\mathrm{Re}(s)|}, \qquad s \in V_\C,
    \end{equation}
    for some $C, \tau > 0$ depending only on $|x|$ and the root system.  Note that (\ref{eqn:carlson-bd}) controls $f$ through $\mathrm{Re}(s)$ alone: along any line in a purely imaginary direction, on which $\mathrm{Re}(s)$ is constant, it bounds $f$ by a constant.

    For $G_{k,s}(x)$, \cite[Proposition 6.1]{Op} gives
    \[
    |G_{k, s}(x)| \; \le \; \sqrt{|W_0|} \, e^{\max_{w \in W_0} \langle \mathrm{Re}(s), w(x) \rangle} \; \le \; \sqrt{|W_0|}\, e^{|x|\,|\mathrm{Re}(s)|}, \qquad s \in V_\C,
    \]
    which is (\ref{eqn:carlson-bd}) with $\tau = |x|$.

    For $\ttE^{(q,k)}_{\mu(x)}(s)$, observe that nonnegativity of the coefficients in the expansion (\ref{eqn:mac-mon-pos-ns}) implies
    \begin{equation} \label{eqn:E-Re-bd}
    |\ttE^{(q,k)}_{\mu(x)}(s)| \;=\; \Bigl|\sum_\nu c_{\mu(x)\nu}\, q^{\langle\nu, s\rangle}\Bigr| \;\leq\; \sum_\nu c_{\mu(x)\nu}\, q^{\langle\nu, \mathrm{Re}(s)\rangle} \;=\; \ttE^{(q,k)}_{\mu(x)}(\mathrm{Re}(s))
    \end{equation}
    for any $s \in V_\C$.  It therefore suffices to bound $\ttE^{(q,k)}_{\mu(x)}(\mathrm{Re}(s))$.

    To do this, we snap $\mathrm{Re}(s)$ to a nearby point of the set $\bar P := \{ \bar \eta : \eta \in P\}$ and then apply the duality (\ref{eqn:mac-nonsym-dual-new}).  Let $D = \sup_{y \in V} \inf_{\eta \in P} |y - \bar \eta|$ denote the covering radius of $\bar P$, which is finite since $P$ has full rank in $V$.  For $y \in V$, fix $\eta(y) \in P$ with $|y - \overline{\eta(y)}| \le D$.

    Writing $\eta = \eta(\mathrm{Re}(s))$, we estimate the ratio
    \[
    \frac{\ttE^{(q,k)}_{\mu(x)}(\mathrm{Re}(s))}{\ttE^{(q,k)}_{\mu(x)}(\bar\eta)} \;=\; \frac{\sum_\nu c_{\mu(x)\nu}\, q^{\langle\nu, \bar\eta\rangle}\, q^{\langle\nu, \mathrm{Re}(s) - \bar\eta\rangle}}{\sum_\nu c_{\mu(x)\nu}\, q^{\langle\nu, \bar\eta\rangle}}.
    \]
    The right-hand side is a weighted average of the positive quantities $q^{\langle\nu, \mathrm{Re}(s) - \bar\eta\rangle}$ with positive coefficients.  Each $\nu$ in the expansion satisfies $|\nu| \le |\mu(x)| \le (|\log q|)^{-1}|x| + O(1)$, so
    \[
    q^{\langle\nu, \mathrm{Re}(s) - \bar\eta\rangle} \;\le\; q^{-|\nu|\,|\mathrm{Re}(s) - \bar\eta|} \;\le\; q^{-((|\log q|)^{-1}|x|+O(1)) D} \;=\; e^{(|x|+o(1)) D},
    \]
    where the $o(1)$ depends only on the root system, and the same bound holds for the weighted average.  We thus obtain
    \begin{equation} \label{eqn:snap-bd}
    \ttE^{(q,k)}_{\mu(x)}(\mathrm{Re}(s)) \;\le\; e^{(|x|+o(1))D}\, \ttE^{(q,k)}_{\mu(x)}(\bar\eta).
    \end{equation}

    Applying the duality (\ref{eqn:mac-nonsym-dual-new}), $\ttE^{(q,k)}_{\mu(x)}(\bar\eta) = \ttE^{(q,k)}_\eta(\overline{\mu(x)})$, and the expansion (\ref{eqn:mac-mon-pos-ns}) gives
    \begin{align*}
    \ttE^{(q,k)}_\eta(\overline{\mu(x)}) \;&=\; c_{\eta\eta}\, q^{\langle \eta,\overline{\mu(x)} \rangle} + \sum_{\nu < \eta} c_{\eta\nu}\, q^{\langle \nu,\overline{\mu(x)} \rangle} \\
    \;&=\; c_{\eta\eta}\, q^{-\langle \eta, \rho \rangle} q^{\langle \eta,\overline{\mu(x)} + \rho \rangle} + \sum_{\nu < \eta} c_{\eta\nu}\, q^{-\langle \nu, \rho \rangle} q^{\langle \nu,\overline{\mu(x)} + \rho \rangle}.
    \end{align*}
    The coefficients $a_\nu := c_{\eta\nu}\,q^{-\langle\nu,\rho\rangle}$ are nonnegative and sum to $\ttE^{(q,k)}_\eta(-\rho) = 1$, so $\ttE^{(q,k)}_\eta(\overline{\mu(x)}) = \sum_\nu a_\nu\, q^{\langle\nu,\,\overline{\mu(x)}+\rho\rangle}$ is a convex combination. Therefore
    \[
    \ttE^{(q,k)}_\eta(\overline{\mu(x)}) \;\le\; \max_\nu q^{\langle\nu,\,\overline{\mu(x)}+\rho\rangle} \;\le\; q^{-|\eta|\,|\overline{\mu(x)}+\rho|},
    \]
    using $\langle\nu,\,\overline{\mu(x)}+\rho\rangle \ge -|\nu|\,|\overline{\mu(x)}+\rho| \ge -|\eta|\,|\overline{\mu(x)}+\rho|$ and $q \in (0,1)$. Since $|\overline{\mu(x)}+\rho| \le (|\log q|)^{-1}|x| + O(1)$, where the $O(1)$ depends only on the root system, we have
    \[
    q^{-|\eta|\,|\overline{\mu(x)}+\rho|} \;\le\; q^{-|\eta|\,((|\log q|)^{-1}|x|+O(1))} \;=\; e^{|\eta|\,(|x|+o(1))}
    \]
    as $q \to 1$.  Substituting $|\eta| \le |\mathrm{Re}(s)| + |\rho| + D$ and combining with (\ref{eqn:snap-bd}) yields the bound
    \begin{equation} \label{eqn:master-bd}
    \ttE^{(q,k)}_{\mu(x)}(\mathrm{Re}(s)) \;\le\; C e^{(|x|+o(1))\,|\mathrm{Re}(s)|}
    \end{equation}
    as $q \to 1$, where $C$ indicates a constant that may differ between expressions but always depends only on $|x|$ and the root system.  By (\ref{eqn:E-Re-bd}) this is (\ref{eqn:carlson-bd}) for $\ttE^{(q,k)}_{\mu(x)}$, with $\tau = |x| + o(1) \le |x| + 1$ for $q$ sufficiently close to $1$.

    Now fix a compact set $K \subset V$.  We next show that for each $\eta \in P$,
    \begin{equation} \label{eqn:lattice-unif}
    \lim_{q \to 1}\, \sup_{x \in K}\, \big| \ttE^{(q,k)}_{\mu(x)}(\bar\eta) - G_{k,\bar\eta}(x) \big| \;=\; 0.
    \end{equation}
    By the duality (\ref{eqn:mac-nonsym-dual-new}) we have $\ttE^{(q,k)}_{\mu(x)}(\bar\eta) = \ttE^{(q,k)}_\eta(\overline{\mu(x)})$, while (\ref{eqn:jac-hgf}) and (\ref{eqn:mac-jac-lim-new}) give
    \[
    \ttE^{(q,k)}_\eta\big((\log q)^{-1} x\big) \to G_{k,\bar\eta}(x)
    \]
    uniformly for $x \in K$.  To prove (\ref{eqn:lattice-unif}), it therefore suffices to show that
    \[
    \big| \ttE^{(q,k)}_\eta\big(\overline{\mu(x)}\big) - \ttE^{(q,k)}_\eta\big((\log q)^{-1} x\big) \big| \;\to \; 0
    \]
    uniformly for $x \in K$.  The two arguments differ by
    \[
    \big| \overline{\mu(x)} - (\log q)^{-1} x \big| \;\le\; c, \qquad x \in K,
    \]
    where $c$ depends only on the root system. Consequently both points $\overline{\mu(x)}$ and $(\log q)^{-1} x$, and the segment joining them, lie in $(\log q)^{-1} K'$ for a fixed compact convex set $K' \supseteq K$ and all $q$ close to $1$. Taking the gradient of the expansion (\ref{eqn:mac-mon-pos-ns}) gives, for $z \in V$,
    \[
    \nabla_z \ttE^{(q,k)}_\eta(z) \;=\; (\log q) \sum_{\nu \le \eta} c_{\eta\nu}\, q^{\langle \nu, z \rangle} \, \nu,
    \]
    and since every weight $\nu$ in the expansion lies in the convex hull of $W_0 \cdot \eta$, so that $|\nu| \le |\eta|$, the nonnegativity of the coefficients yields
    \[
    \big| \nabla_z \ttE^{(q,k)}_\eta(z) \big| \;\le\; |\log q|\, |\eta| \sum_{\nu \le \eta} c_{\eta\nu}\, q^{\langle \nu, z \rangle} \;=\; |\log q|\, |\eta|\, \ttE^{(q,k)}_\eta(z).
    \]
    For $z$ on the segment joining $\overline{\mu(x)}$ and $(\log q)^{-1} x$ we have $z = (\log q)^{-1}\zeta$ with $\zeta \in K'$, and by (\ref{eqn:mac-jac-lim-new}) the values $\ttE^{(q,k)}_\eta\big((\log q)^{-1}\zeta\big)$ converge uniformly for $\zeta \in K'$ and thus are bounded by a constant $C = C(\eta, K')$ for $q$ close to $1$.  The mean value inequality then gives
    \[
    \big| \ttE^{(q,k)}_\eta\big(\overline{\mu(x)}\big) - \ttE^{(q,k)}_\eta\big((\log q)^{-1} x\big) \big|
    \;\le\; c\, |\log q|\, |\eta|\, C,
    \]
    which tends to $0$ uniformly for $x \in K$ as $q \to 1$, proving (\ref{eqn:lattice-unif}).

    We now combine (\ref{eqn:lattice-unif}) with the growth estimate (\ref{eqn:carlson-bd}).  Write $g_q(s,x) = \ttE^{(q,k)}_{\mu(x)}(s) - G_{k,s}(x)$.  As shown above, both terms satisfy (\ref{eqn:carlson-bd}) with constants $C, \tau$ that are bounded above by continuous functions of $|x|$, so there are constants $C > 0$ and $\tau > 0$ depending only on $K$ and the root system such that
    \begin{equation} \label{eqn:g-bds}
    |g_q(s,x)| \;\le\; C e^{\tau |\mathrm{Re}(s)|}, \qquad s \in V_\C,
    \end{equation}
    for all $x \in K$ and $q$ close to $1$.

    Now let $K_s \subset V_\C$ be compact and suppose, for contradiction, that $\ttE^{(q,k)}_{\mu(x)}(s)$ does not converge uniformly to $G_{k,s}(x)$ on $K_s \times K$ as $q \to 1$.  Then there exist $\varepsilon > 0$, a sequence $q_j \to 1$, and points $(s_j, x_j) \in K_s \times K$ with $|g_{q_j}(s_j, x_j)| \ge \varepsilon$. Passing to a subsequence, we may suppose $s_j \to s_*$ and $x_j \to x_*$.  For each fixed $x_j$ the function $h_j := g_{q_j}(\,\cdot\,, x_j)$ is entire, and by (\ref{eqn:g-bds}) the family $(h_j)_{j \ge 1}$ is locally uniformly bounded on $V_\C$.  Therefore, by Montel's theorem, a subsequence converges locally uniformly to an entire function $h_*$, which satisfies the same bound (\ref{eqn:g-bds}).  By (\ref{eqn:lattice-unif}), $|h_j(\bar\eta)| \le \sup_{x \in K} |g_{q_j}(\bar\eta, x)| \to 0$ for each $\eta \in P$, so $h_*$ vanishes on $\bar P$.

    We claim this forces $h_* \equiv 0$.  Identify $V_\C \cong \C^n$ via the fundamental weight basis, and write a point as $z = (z_1, \dots, z_n)$.  The strictly dominant weights are exactly the $\eta = \sum_i a_i \omega_i$ with every $a_i$ a positive integer, and for these points $\bar\eta = \eta + \rho$; thus $h_*$ vanishes on the grid $\rho + \Z_{\ge 1}^n$.

    Fix the coordinates $z_2, \dots, z_n$ and regard $h_*$ as an entire function of the single variable $z_1$.  Along the imaginary $z_1$-axis the real parts of $z_2, \dots, z_n$ are fixed and $\mathrm{Re}(z_1) = 0$, so $\mathrm{Re}(s)$ is constant. Thus by (\ref{eqn:g-bds}) the function $z_1 \mapsto h_*(z_1, z_2, \dots, z_n)$ is bounded on the imaginary axis, and the bound $|h_*(z_1, z_2, \dots, z_n)| \le C e^{\tau|\mathrm{Re}(s)|} \le C e^{\tau(|z_1| + \cdots + |z_n|)}$ shows it has exponential type at most $\tau$.  If moreover $z_2, \dots, z_n$ are taken on the grid, $z_i = \rho_i + a_i$ with $a_i \in \Z_{\ge 1}$, then $z_1 \mapsto h_*$ vanishes at $\rho_1 + a_1$ for every $a_1 \in \Z_{\ge 1}$, so Carlson's theorem gives $h_*(\,\cdot\,, z_2, \dots, z_n) \equiv 0$.  Since this holds for all grid values of $z_2, \dots, z_n$, the function $h_*$ vanishes identically in $z_1$ whenever $z_2, \dots, z_n$ lie on the grid, with $z_1$ now \emph{arbitrary} in $\C$.

    We now repeat the above argument in each successive coordinate.  The bound (\ref{eqn:g-bds}) depends only on $\mathrm{Re}(s)$, so it continues to give boundedness on each successive one-dimensional imaginary axis even when the remaining coordinates are fixed at arbitrary complex values, as required at each iteration.  After $n$ applications of Carlson's theorem, we conclude $h_* \equiv 0$.

    But $h_j \to h_*$ locally uniformly and $s_j \to s_*$, so $h_j(s_j) \to h_*(s_*) = 0$, contradicting $|h_j(s_j)| \ge \varepsilon$.  Thus (\ref{eqn:mac-hgf}) holds and the convergence is uniform for $(s,x)$ in compact subsets of $V_\C \times V$, completing the proof.
\end{proof}

The following theorem is the main result of this paper.

\begin{thm} \label{thm:GLP} 
    For each $x \in V$ and nonnegative multiplicity parameter $k$, there is a unique probability measure $\pi_{k,x}$ supported on the convex hull of the orbit $W_0 \cdot x$ such that
    \begin{equation} \label{eqn:GLP}
        G_{k,s-\rho}(x) = \int_V e^{\langle s, y \rangle} \, \pi_{k, x}(dy), \qquad s \in V_\C.
    \end{equation}
\end{thm}

\begin{proof}
Since $G_{k,s}(x)$ is holomorphic in $s$, it suffices to consider $s \in V$. We first suppose that $k$ is strictly positive and then remove this assumption at the end of the proof. Observe that from (\ref{eqn:mac-mon-pos-ns}) we can write
    \begin{equation} \label{eqn:E-int}
    \ttE^{(q,k)}_{\mu}(s-\rho) = \int_V e^{\langle s, y \rangle} \varpi_{q, \mu}(dy),
    \end{equation}
    where
    \begin{equation} \label{eqn:varpi-def}
    \varpi_{q, \mu} = \sum_{\nu \le \mu} c_{\mu\nu} q^{-\langle \nu, \rho \rangle} \delta_{(\log q) \nu}
    \end{equation}
    with $c_{\mu \nu}$ defined as in (\ref{eqn:mac-mon-pos-ns}).  Note that $\varpi_{q, \mu}$ is a probability measure, as each coefficient $c_{\mu\nu} q^{-\langle \nu, \rho \rangle}$ is positive, and
    \[
    \int_V \varpi_{q, \mu}(dy) = \ttE^{(q,k)}_{\mu}(-\rho) = 1.
    \]
    Moreover, since the exponentials appearing in the expansion of $\ttE^{(q,k)}_\mu$ are contained in the convex hull $\mathrm{conv}(W_0 \cdot \mu)$, the support of $\varpi_{q, \mu}$ is contained in $\mathrm{conv}(W_0 \cdot (\log q) \mu)$.
    
    From Theorem \ref{thm:mac-hgf} and (\ref{eqn:E-int}), we have
    \[
    G_{k,s-\rho}(x) = \lim_{q \to 1} \int_V e^{\langle s, y \rangle} \varpi_{q, \mu(x)}(dy)
    \]
    with $\mu(x)$ as defined in Theorem \ref{thm:mac-hgf}.  As $q \to 1$ the supports of the measures $\varpi_{q, \mu(x)}$ are eventually contained in $\mathrm{conv}(W_0 \cdot (1 + \varepsilon)x)$ for any $\varepsilon > 0$.  In other words, $G_{k,s-\rho}(x)$ is a pointwise limit of Laplace transforms (moment generating functions) of probability measures with uniformly bounded support.  By standard results on convergence of moment generating functions (see e.g. \cite[\S30]{Billingsley1995}), this implies that the measures $\varpi_{q, \mu}$ converge weakly and in all moments to a probability measure $\pi_{k, x}$ supported on $\cap_{\varepsilon > 0} \mathrm{conv}(W_0 \cdot (1 + \varepsilon)x) =\mathrm{conv}(W_0 \cdot x)$, and moreover $G_{k,s-\rho}(x)$ is the Laplace transform of $\pi_{k, x}$.  Since a compactly supported measure on $V$ is determined by its Laplace transform, uniqueness of $\pi_{k, x}$ is immediate, and we have shown the desired result for strictly positive $k$.

    The statement for all nonnegative $k$ follows by letting any of the values of $k$ tend to $0$ in (\ref{eqn:GLP}). Since the map $k \mapsto G_{k,s-\rho}(x)$ is continuous, a similar argument by convergence of Laplace transforms proves the claim. 
    \end{proof}

    \begin{cor} \label{cor:positivity}
        The trigonometric Dunkl intertwining operator $\mathcal{V}_k$ is positivity preserving.  That is, if $f \in C^\infty(V)$ and $f(x) \ge 0$ for all $x \in V$, then $\mathcal{V}_k f(x) \ge 0$ for all $x \in V$.
    \end{cor}

    \begin{proof}
        Let $\mathcal{I}$ be the operator on $C^\infty(V)$ defined by
        \[
        \mathcal{I} f(x) = \int_V f(y) \, e^{\langle \rho, y \rangle} \pi_{k, x}(dy).
        \]
        Clearly $\mathcal{I}$ is  positivity preserving since it is defined by integration against the positive measure $e^{\langle \rho, y \rangle} \pi_{k, x}(dy)$. We shall show that $\mathcal{I} = \mathcal{V}_k$.

        Recall that $\mathcal{V}_k$ is continuous by definition, while $\mathcal{I}$ is also continuous, again because it is defined by integration against a positive measure.  To show that $\mathcal{I} = \mathcal{V}_k$, it therefore suffices to show that the two operators coincide on a dense subset of $C^\infty(V)$. Theorem \ref{thm:GLP} and (\ref{eqn:G-V-exp}) give
        \[
        \mathcal{I} (e^{\langle s, \,\cdot\,\rangle})(x) = G_{k, s}(x) = \mathcal{V}_k(e^{\langle s, \,\cdot\,\rangle})(x),
        \]
        so that $\mathcal{I}$ and $\mathcal{V}_k$ coincide on exponential polynomials, which are dense in $C^\infty(V)$. This completes the proof.
    \end{proof}

\begin{remark} \label{rem:rational}
The \emph{Dunkl kernel} is obtained from the Opdam--Cherednik kernel via the \emph{rational limit} (see \cite[\S3.2 and \S4.4]{AnkerDunklNotes}):
\[
E_{k,s}(x) = \lim_{\varepsilon \to 0} G_{k, \varepsilon^{-1}s}(\varepsilon x).
\]
Taking the rational limit in Theorem \ref{thm:GLP} and dilating the domain of the integral to absorb the $\varepsilon^{-1}$ scaling in the exponent, we obtain a new, short proof of the fact, shown in \cite{RoslerPositivity}, that $E_{k,s}(x)$ is the Laplace transform of a probability measure supported on $\mathrm{conv}(W_0 \cdot x)$, and as a consequence the \emph{rational} Dunkl intertwining operator is positivity preserving.
\end{remark}

\section{Applications}

\subsection{Spectral log-convexity and majorization inequalities} \label{sec:majorization}

In \cite{McSwiggenSahiMajorization}, the authors proved that the Heckman--Opdam hypergeometric function of type $A$ is a log-convex function of its spectral parameter, and that the same holds for the symmetric Macdonald polynomials of type $A$ when evaluated at certain lattice points. As a consequence, the type-$A$ symmetric Macdonald polynomials, Heckman--Opdam hypergeometric functions, and Jack polynomials satisfy Muirhead-type majorization inequalities.  Here we generalize the results of \cite{McSwiggenSahiMajorization} to the nonsymmetric setting and to arbitrary root systems, with considerably simpler proofs.

We first show three results, Propositions \ref{prop:E-lc} and \ref{prop:conv-mac} and Theorem \ref{thm:maj-mac}, that generalize statements for Macdonald polynomials shown in \cite{McSwiggenSahiMajorization}. As written these results apply to any of the self-dual cases of Section \ref{sec:prelim-mac}, though corresponding statements in the non-self-dual cases follow by the same arguments.  Then in Corollaries \ref{cor:G-lc} and \ref{cor:conv-HGF} and Theorem \ref{thm:HGF-maj}, we show analogous results for the Opdam--Cherednik kernel and Heckman--Opdam hypergeometric function, using the Laplace transform representation of Theorem \ref{thm:GLP}.

Let $X$ be a convex subset of a Euclidean space $V$. A function $f: X \to [0,\infty)$ is \emph{log-convex} if $\log \circ f$ is convex, with the convention $\log 0 = -\infty$.  Equivalently, for all $x, y \in X$ and all $t \in (0,1)$, we have
\[
f \big( tx + (1-t)y \big) \le f(x)^t f(y)^{1-t}.
\]
If $X$ is an arbitrary subset of $V$, not necessarily convex, we say that a function $f$ on $X$ is (log-)convex if $f$ extends to a (log-)convex function on a convex set $X' \supset X$.

\begin{prop} \label{prop:E-lc}
    For all $\mu \in P$ and $w \in W_0$, the map $\eta \mapsto \ttE^{(q,k)}_\eta(\bar \mu)$ is a log-convex function on the set $P_w = \{ \eta \in P : w_\eta = w\}$.
\end{prop}

\begin{proof}
From the duality (\ref{eqn:mac-nonsym-dual-new}) and the positive expansion (\ref{eqn:mac-mon-pos-ns}), we have
\[
\ttE^{(q,k)}_\eta(\bar \mu) \;=\; \ttE^{(q,k)}_\mu(\bar \eta) \;=\; \sum_{\nu \le \mu} c_{\mu\nu}\, q^{\langle \nu,\bar \eta \rangle} \;=\; \sum_{\nu \le \mu} c_{\mu\nu}\, q^{\langle \nu,\eta - w_\eta^{-1}(\rho) \rangle}.
\]
For all $\eta \in P_w$, the shift $w_\eta^{-1}(\rho) = w^{-1}(\rho)$ is a constant vector independent of $\eta$.  Accordingly, the sum above is the restriction to $P_w$ of a positive linear combination of log-affine (and thus log-convex) functions on $V$, and it is therefore a log-convex function on $P_w$.
\end{proof}

\begin{remark}
The restriction to $P_w$ in Proposition \ref{prop:E-lc} is necessary: the analogous statement on all of $P$ is false. In Macdonald's case \cite[(1.4.1)]{MacdonaldAffine} with $R = A_1$, taking $q = 1/2$ and $k = 1$ and identifying $P = \Z\omega$ where $\omega$ is the fundamental weight, a direct computation from \cite[(6.2.7)--(6.2.8)]{MacdonaldAffine} gives
\[
\ttE^{(q,k)}_{-\omega}(\bar\omega) \;=\; \sqrt 2, \qquad \ttE^{(q,k)}_0(\bar\omega) \;=\; 1, \qquad \ttE^{(q,k)}_\omega(\bar\omega) \;=\; \tfrac{1}{2\sqrt 2},
\]
which violates the necessary inequality $\ttE^{(q,k)}_0(\bar\omega)^2 \le \ttE^{(q,k)}_{-\omega}(\bar\omega) \cdot \ttE^{(q,k)}_\omega(\bar\omega)$ for log-convexity to hold at the midpoint $\eta = 0$ of $\{-\omega, \omega\}$. The obstruction arises because the map $\eta \mapsto \bar\eta = \eta - w_\eta^{-1}(\rho)$ is only piecewise affine on $P$, with jumps at chamber walls. On any fixed $P_w$, however, $\bar\eta$ is affine in $\eta$, and log-convexity holds.
\end{remark}

Given $x, y \in V$ and a linear action of a group $G$ on $V$, we say that $x$ $G$-\emph{majorizes} $y$ if $y$ is contained in the convex hull of the $G$-orbit of $x$.  We say that a function $f : \Omega \to \R$, $\Omega \subseteq V$ is $G$-\emph{convex} if $f(x) \ge f(y)$ whenever $x$ $G$-majorizes $y$.  A $G$-invariant, convex function is $G$-convex \cite[Theorem 1]{GW}.

\begin{prop} \label{prop:conv-mac}
    For all $\mu \in P_+$, the map $\lambda \mapsto \ttP^{(q,k)}_\lambda(\bar \mu)$ is a log-convex and $W_0$-convex function on $P_+$.
\end{prop}

\begin{proof}
    We first prove log-convexity. 
    Analogously to the preceding proof, the duality (\ref{eqn:mac-sym-dual-new}) and the positive expansion (\ref{eqn:mac-mon-pos-sym}) give
    \begin{align*}
    \ttP^{(q,k)}_\lambda(\bar \mu) \;=\; \ttP^{(q,k)}_\mu(\bar \lambda) \;&=\; b_{\mu\mu}\, m_\mu(\bar \lambda) + \sum_{\nu < \mu} b_{\mu\nu}\, m_\nu(\bar \lambda) \\
    \;&=\; \sum_{w \in W_0} \bigg( \frac{|W_0 \cdot \mu|}{|W_0|} b_{\mu\mu}\, q^{\langle w(\mu),\bar \lambda \rangle} + \sum_{\nu < \mu} \frac{|W_0 \cdot \nu|}{|W_0|} b_{\mu\nu}\, q^{\langle w(\nu),\bar \lambda \rangle} \bigg)
    \end{align*}
    which is a positive sum of log-convex functions in $\lambda$ and is therefore log-convex.

    Next, we prove $W_0$-convexity. Define
    \begin{equation*}
        f(x) \; = \; \sum_{w \in W_0} \bigg( \frac{|W_0 \cdot \mu|}{|W_0|} b_{\mu\mu}\, q^{\langle w(\mu),x \rangle} + \sum_{\nu < \mu} \frac{|W_0 \cdot \nu|}{|W_0|} b_{\mu\nu}\, q^{\langle w(\nu),x \rangle} \bigg), \qquad x \in V,
    \end{equation*}
    so that $f(\bar \lambda) = \ttP^{(q,k)}_\lambda(\bar \mu)$. Then $f$ is manifestly both log-convex (thus convex) and also $W_0$-invariant, and is therefore $W_0$-convex.
    
     Since the $W_0$-majorization order is defined by orbit containment under a linear group action, it is invariant under translations, so that for $\lambda, \nu \in P_+$, we have $\lambda$ $W_0$-majorizes $\nu$ if and only if $\lambda + \rho$ $W_0$-majorizes $\nu + \rho$.  Moreover for $\lambda \in P_+$ we have $\bar \lambda \in W_0(\lambda + \rho)$, so in fact $\lambda$ $W_0$-majorizes $\nu$ if and only if $\bar \lambda$ $W_0$-majorizes $\bar \nu$.  Therefore, since $f$ is $W_0$-convex, if $\lambda$ $W_0$-majorizes $\nu$ then $f(\bar \lambda) = \ttP^{(q,k)}_\lambda(\bar \mu) \ge f(\bar \nu) = \ttP^{(q,k)}_\nu(\bar \mu)$, which is the desired statement that $\lambda \mapsto \ttP^{(q,k)}_\lambda(\bar \mu)$ is $W_0$-convex.
\end{proof}

\begin{thm} \label{thm:maj-mac}
    For all $\lambda, \mu \in P_+$, the following are equivalent:
    \begin{enumerate}
        \item $\lambda$ $W_0$-majorizes $\mu$.
        \item For all choices of parameters $q \in (0,1)$ and $t_a \in (0,1)$ $(a \in S)$, $\ttP^{(q,k)}_\lambda(\bar \nu) \ge \ttP^{(q,k)}_\mu(\bar \nu)$ for all $\nu \in P_+$.
        \item There exists a choice of parameters $q \in (0,1)$ and $t_a \in (0,1)$ $(a \in S)$ such that $\ttP^{(q,k)}_\lambda(\bar \nu) \ge \ttP^{(q,k)}_\mu(\bar \nu)$ for all $\nu \in P_+$.
    \end{enumerate}
\end{thm}

\begin{proof}
    The implication (1) $\implies$ (2) is the $W_0$-convexity statement of Proposition \ref{prop:conv-mac}, and (2) $\implies$ (3) is trivial. It remains to show (3) $\implies$ (1). We prove the contrapositive: assuming that $\lambda$ does not $W_0$-majorize $\mu$, we find a sequence $\nu_t \in P_+$ along which $\ttP^{(q,k)}_\mu(\bar\nu_t) > \ttP^{(q,k)}_\lambda(\bar\nu_t)$ for any fixed $(q,k)$ and $t$ sufficiently large, contradicting (3).

    Since $\mu \notin \mathrm{conv}(W_0 \cdot \lambda)$, by the hyperplane separation theorem there exists $x \in V$ with $\langle x, \mu\rangle > \max_{w \in W_0} \langle x, w(\lambda)\rangle$. We may assume $x$ is dominant as the right-hand side of this inequality is $W_0$-invariant in $x$, so $\max_w \langle x, w(\lambda)\rangle = \langle x, \lambda\rangle$, giving $\langle x, \mu - \lambda\rangle > 0$. Set $y = -w_0 (x)$, which is also dominant since $-w_0$ fixes the dominant chamber, and take $\nu_t = \lfloor t y \rfloor \in P_+$ for $t > 0$, where the floor function is applied coordinate-wise in the fundamental weight basis, so that $\nu_t / t \to y$ as $t \to \infty$.

    We estimate the leading behavior of $\ttP^{(q,k)}_\lambda(\bar\nu_t)$ as $t \to \infty$ from the expansion (\ref{eqn:mac-mon-pos-sym}):
    \begin{equation} \label{eqn:P-nut-sum}
    \ttP^{(q,k)}_\lambda(\bar\nu_t) \;=\; \sum_{\nu \leq \lambda} b_{\lambda\nu}\, m_\nu(\bar\nu_t) \;=\; \sum_{\nu \leq \lambda} b_{\lambda\nu} \frac{|W_0 \cdot \nu|}{|W_0|} \sum_{w \in W_0} q^{\langle w(\nu), \bar\nu_t\rangle}.
    \end{equation}
    Substituting $\bar\nu_t = \nu_t + O(1)$ and $\nu_t = ty + O(1)$, we have
    \[
    \langle w(\nu), \bar\nu_t\rangle \;=\; t \langle w(\nu), y\rangle + O(1)
    \]
    for all $\nu \le \lambda$ and $w \in W_0$. The leading term in the expansion (\ref{eqn:P-nut-sum}) as $t \to \infty$ thus corresponds to the smallest value of $\langle w(\nu), y\rangle$. For $y$ dominant and $\nu \in P_+$, the minimum of $\langle w(\nu), y\rangle$ over $w \in W_0$ is attained at $w = w_0$ (sending $\nu$ to the antidominant representative of its orbit), and using $w_0^{-1} = w_0$ and $-w_0 (y) = x$, this minimum equals $\langle w_0(\nu), y\rangle = -\langle \nu, x\rangle$. Since $\nu \leq \lambda$ in the dominance order and $x$ is dominant, $\langle \lambda - \nu, x\rangle \geq 0$, giving $\langle \nu, x\rangle \leq \langle \lambda, x\rangle$. The minimum of $\langle w(\nu), y\rangle$ over all $\nu \le \lambda$ and $w \in W_0$ is therefore $-\langle \lambda, x\rangle$, attained at $\nu = \lambda$, $w = w_0$.

    Since $q \in (0,1)$ and all terms in the sum are positive, this largest term (with $\nu = \lambda$, $w = w_0$) provides a positive lower bound on $\ttP^{(q,k)}_\lambda(\bar\nu_t)$, and the sum is bounded above by the number of terms $|W_0|(\# \{\nu \le \lambda\})$ times this same largest term. Combining these bounds gives
    \[
    \log_q \ttP^{(q,k)}_\lambda(\bar\nu_t) \;=\; -t \langle \lambda, x\rangle + O(1)
    \]
    as $t \to \infty$, where the $O(1)$ depends on $\lambda$, the root system, $q$ and $k$, but not on $t$. The analogous estimate for $\ttP^{(q,k)}_\mu(\bar\nu_t)$ gives
    \[
    \log_q \ttP^{(q,k)}_\mu(\bar\nu_t) \;=\; -t \langle \mu, x\rangle + O(1).
    \]
    
    Taking the difference of the two preceding displays, we obtain
    \[
    \log_q \frac{\ttP^{(q,k)}_\mu(\bar\nu_t)}{\ttP^{(q,k)}_\lambda(\bar\nu_t)} \;=\; -t \langle \mu - \lambda, x\rangle + O(1).
    \]
    Since $\log q < 0$ and $\langle \mu - \lambda, x\rangle > 0$,
    \[
    \log \frac{\ttP^{(q,k)}_\mu(\bar\nu_t)}{\ttP^{(q,k)}_\lambda(\bar\nu_t)} \;=\; |\log q|\, t\, \langle \mu - \lambda, x\rangle + O(1) \;\to\; \infty
    \]
    as $t \to \infty$, so $\ttP^{(q,k)}_\mu(\bar\nu_t) > \ttP^{(q,k)}_\lambda(\bar\nu_t)$ for $t$ sufficiently large, contradicting (3).
\end{proof}

\begin{remark}
    Analogous statements to Propositions \ref{prop:E-lc} and \ref{prop:conv-mac} and Theorem \ref{thm:maj-mac} hold in the non-self-dual cases discussed in Section \ref{sec:prelim-mac} for both the Macdonald polynomials and the dual polynomials. We omit these here as the proofs are identical but the notation becomes significantly more cumbersome.
\end{remark}

We now consider the setting of Section \ref{sec:prelim-OC} and fix a root system $\Phi \subset V$.  The next two results are almost immediate from Theorem \ref{thm:GLP}.

\begin{cor} \label{cor:G-lc}
    For all nonnegative multiplicity parameters $k$ and all $x \in V$, the map $s \mapsto G_{k,s}(x)$ is a log-convex function on $V$.
\end{cor}

\begin{proof}
    Theorem \ref{thm:GLP} gives
    \[
    G_{k,s}(x) = \int_V e^{\langle s + \rho, y \rangle} \pi_{k, x}(dy).
    \]
    This is a positive linear combination (integral) of log-convex functions of $s$ and is therefore log-convex.
\end{proof}

\begin{cor} \label{cor:conv-HGF}
    For all nonnegative multiplicity parameters $k$ and all $x \in V$, the map $s \mapsto F_{k,s}(x)$ is a log-convex and $W_0$-convex function on $V$.
\end{cor}

\begin{proof}
    By Corollary \ref{cor:G-lc} and the definition (\ref{eqn:HGF-def}), $s \mapsto F_{k,s}(x)$ is a positive sum of log-convex functions and is therefore log-convex.  Since $F_{k,s}(x)$ is also $W_0$-invariant in $s$, it is $W_0$-convex.
\end{proof}

The following theorem resolves a conjecture by McSwiggen and Novak \cite[Conjecture 4.7]{MN-majorization}.

\begin{thm} \label{thm:HGF-maj}
    For any $r, s \in V$, the following are equivalent:
    \begin{enumerate}
        \item $r$ $W_0$-majorizes $s$.
        \item  For all nonnegative multiplicity parameters $k$, $F_{k, r}(x) \ge F_{k,s}(x)$ for all $x \in V$.
        \item There exists a nonnegative multiplicity parameter $k$ such that $F_{k, r}(x) \ge F_{k,s}(x)$ for all $x \in V$.
    \end{enumerate}
\end{thm}

\begin{proof}
    The implication (1) $\implies$ (2) is the $W_0$-convexity statement of Corollary \ref{cor:conv-HGF}.  The implication (2) $\implies$ (3) is trivial. The implication (3) $\implies$ (1) is shown in \cite[Proposition 4.8]{MN-majorization}.
\end{proof}

\begin{remark}
When $r = \bar \lambda$, $s = \bar \mu$ for $\lambda, \mu \in P_+$, Theorem \ref{thm:HGF-maj} combined with the specialization (\ref{eqn:jac-hgf}) gives analogous majorization inequalities for the symmetric Heckman--Opdam Jacobi polynomials, generalizing the inequalities for Jack polynomials shown in \cite{McSwiggenSahiMajorization}.
\end{remark}

\subsection{Further prospects}
\label{sec:appl-analysis}

The Laplace transform representation in Theorem \ref{thm:GLP} offers a new approach to analytic questions in trigonometric Dunkl theory by translating statements about the kernel $G_{k,s}$ into statements about the measures $\pi_{k,x}$.  In principle, information about $\pi_{k,x}$ is accessible via the limit from the measures $\varpi_{q,\mu}$ defined in (\ref{eqn:varpi-def}), whose coefficients are given explicitly by the Ram--Yip formula \cite[Theorem 3.1]{RamYip2011}. This possibility opens a new line of attack on a number of problems that have so far remained intractable.

Although such applications are beyond the scope of the present work, we briefly illustrate the method in hopes that it will be useful to others. The most direct and elementary example is that regularity results on $\pi_{k,x}$ would translate immediately, via classical techniques of Fourier analysis, to sharper decay estimates on $G_{k,s}$ for imaginary spectral parameters. Here we consider a slightly more sophisticated application to a long-standing open problem in the field: sharp estimates for the Heckman--Opdam heat kernel.  Throughout this subsection we work in the setting of Section \ref{sec:prelim-OC}, fixing an irreducible root system $\Phi$ spanning $V \cong \R^n$ together with a strictly positive multiplicity parameter $k$.

We briefly recall the construction of the Heckman--Opdam heat kernel, following Schapira \cite{SchapiraHGF, SchapiraHOProc}.  Let
\[
L_k \;=\; \sum_{i=1}^n D_{k,\xi_i}^2
\]
be the \emph{Heckman--Opdam Laplacian}, where $\{\xi_1, \dots, \xi_n\}$ is any orthonormal basis of $V$; the operator $L_k$ is independent of the chosen basis.  Let $\mu_k$ be the measure on $V$ defined by $\mu_k(dx) = \delta_k(x)\,dx$, where
\[
\delta_k(x) \;=\; \prod_{\alpha \in \Phi^+} \Bigl| \sinh \tfrac{\langle \alpha, x \rangle}{2} \Bigr|^{2k(\alpha)},
\]
and let $\mathcal{D}_k = \tfrac{1}{2}(L_k - |\rho|^2)$.  As shown in \cite{SchapiraHGF}, the closure of $\mathcal{D}_k$ generates a strongly continuous Feller semigroup $(P_t^{(k)})_{t \geq 0}$, given for $t > 0$ by integration with respect to $\mu_k$ against the \emph{heat kernel}
\begin{equation} \label{eqn:heat-kernel-def}
p_t^{(k)}(x,y) \;=\; \int_V e^{-\frac{t}{2}\left(|\xi|^2 + |\rho|^2\right)}\, G_{k, i\xi}(x)\, G_{k, i\xi}(-y)\, \nu_k(d\xi), \qquad x, y \in V,
\end{equation}
where $\nu_k$ is the asymmetric Plancherel measure on $V \cong i V$ \cite[\S4]{SchapiraHGF}:
\begin{equation} \label{eqn:asym-plancherel}
  \nu_k(d\xi) \;=\; \mathfrak{c}\prod_{\alpha \in \Phi^+}
  \frac{\Gamma\!\big(i\langle \xi, \alpha^\vee\rangle + k(\alpha) + \tfrac{1}{2}k(\tfrac{\alpha}{2})\big)\,
        \Gamma\!\big(-i\langle \xi, \alpha^\vee\rangle + k(\alpha) + \tfrac{1}{2}k(\tfrac{\alpha}{2}) + 1\big)}
       {\Gamma\!\big(i\langle \xi, \alpha^\vee\rangle + \tfrac{1}{2}k(\tfrac{\alpha}{2})\big)\,
        \Gamma\!\big(-i\langle \xi, \alpha^\vee\rangle + \tfrac{1}{2}k(\tfrac{\alpha}{2}) + 1\big)}\, d\xi,
\end{equation}
where $\mathfrak{c}$ is a normalizing constant and $k(\tfrac{\alpha}{2}) = 0$ if $\tfrac{\alpha}{2} \notin \Phi$. The heat kernel $p_t^{(k)}$ is the transition kernel of the \emph{Heckman--Opdam process}, the c\`adl\`ag Feller process on $V$ with generator $\mathcal{D}_k$ \cite{SchapiraHOProc}.

Here we derive an alternative integral formula for the heat kernel via the Laplace transform representation of Theorem \ref{thm:GLP}.

\begin{prop} \label{prop:heat-rep}
    Define the kernel
    \begin{equation} \label{eqn:theta-def}
    \Theta_t(z) \;=\; \int_V e^{-\frac{t}{2}|\xi|^2}\, e^{i \langle \xi, z \rangle}\,\nu_k(d \xi), \qquad z \in V,\ t > 0.
    \end{equation}
    For all $t > 0$ and all $x, y \in V$,
    \begin{equation} \label{eqn:heat-rep}
    p_t^{(k)}(x,y) \;=\; e^{-\frac{t}{2}|\rho|^2} \int_V e^{\langle \rho,\, z \rangle}\, \Theta_t(z)\, \big( \pi_{k,x} * \pi_{k,-y} \big) (dz),
    \end{equation}
    where $*$ indicates the convolution of measures on $V \cong \R^n$, and the probability measures $\pi_{k,x}$, $\pi_{k,-y}$ are defined as in Theorem \ref{thm:GLP}. 
\end{prop}

\begin{proof}
    For an imaginary spectral parameter $s = i\xi$, Theorem \ref{thm:GLP} gives
    \begin{equation} \label{eqn:G-fourier}
    G_{k,i\xi}(x) \;=\; \int_V e^{i \langle \xi, y \rangle}\, e^{\langle \rho, y \rangle}\, \pi_{k,x}(dy),
    \end{equation}
    so that $G_{k,i\xi}(x)$ is the Fourier transform of the finite measure $e^{\langle \rho, y \rangle} \pi_{k,x}(dy)$ supported on the compact set $\mathrm{conv}(W_0 \cdot x)$.  Inserting (\ref{eqn:G-fourier}) for both factors $G_{k,i\xi}(x)$ and $G_{k,i\xi}(-y)$ into the heat kernel formula (\ref{eqn:heat-kernel-def}), we obtain
    \[
    p_t^{(k)}(x,y) = \int_V e^{-\frac{t}{2}(|\xi|^2 + |\rho|^2)} \left( \int_V e^{\langle \rho, y' \rangle} e^{i\langle \xi, y' \rangle} \pi_{k,x}(dy') \right) \left( \int_V e^{\langle \rho, z \rangle} e^{i\langle \xi, z \rangle} \pi_{k,-y}(dz) \right) \nu_k(d\xi).
    \]
    The above triple integral is absolutely convergent: the integrand is bounded in modulus by
    \[
    e^{-\frac{t}{2}|\xi|^2}\, e^{\langle \rho, y' \rangle}\, e^{\langle \rho, z \rangle},
    \]
    the exponentials $e^{\langle \rho, y' \rangle}$ and $e^{\langle \rho, z \rangle}$ are bounded on the compact supports of $\pi_{k,x}$ and $\pi_{k,-y}$, and the total variation $|\nu_k|$ satisfies
    \[
    \int_V e^{-\frac{t}{2}|\xi|^2}\, |\nu_k|(d \xi) < \infty.
    \]
    To see this last claim, observe that after applying Stirling's approximation to the Gamma functions in (\ref{eqn:asym-plancherel}), the exponential factors cancel between the numerator and denominator, which shows that the asymmetric Plancherel density grows at most polynomially at infinity.  We may therefore apply Fubini's theorem to integrate first with respect to $\xi$, which gives $\Theta_t(y' + z)$ in the integrand, and then rewriting one of the remaining integrals as a convolution gives (\ref{eqn:heat-rep}).
\end{proof}

Schapira proved that $p_t^{(k)}$ is strictly positive and symmetric \cite[Corollaries 5.1 and 5.2]{SchapiraHGF}, and obtained a sharp two-sided estimate for $p_t^{(k)}(0,x)$ \cite[Theorem 5.2]{SchapiraHGF}.  However, a corresponding sharp estimate on $p_t^{(k)}(x,y)$ for general $x, y$ has remained out of reach.  The difficulty is that the only available pointwise control of the integrand in (\ref{eqn:heat-kernel-def}) is the magnitude bound $|G_{k,i\xi}(x)| \le G_{k,0}(x)$ \cite[Proposition 3.1]{SchapiraHGF}, which discards the oscillation of $G_{k,i\xi}$ for imaginary spectral parameters, but that oscillation is precisely what determines the joint dependence of $p_t^{(k)}$ on $x$ and $y$.

To see how this difficulty could be resolved with a better understanding of the measures $\pi_{k,x}$, consider how a two-sided estimate on $p_t^{(k)}(x,y)$ would follow from (\ref{eqn:heat-rep}). As $t \to 0$, the Gaussian factor in (\ref{eqn:theta-def}) tends to $1$, and integration against $\Theta_t$ localizes around $z = 0$. The small-$t$ behavior of the integral in (\ref{eqn:heat-rep}) is therefore determined by the mass that the convolution $\pi_{k,x} * \pi_{k,-y}$ places in a shrinking neighborhood of the origin. Specifically, suppose that near $z = 0$ the measure $\pi_{k,x} * \pi_{k,-y}$ were absolutely continuous with a density $g_{x,y}(z)$ satisfying a two-sided bound $c \le g_{x,y}(z) \le C$, with constants depending on $x$ and $y$. Laplace's method applied to (\ref{eqn:heat-rep}) would then yield matching upper and lower bounds on $p_t^{(k)}(x,y)$.

\section*{Acknowledgements}

C.M. would like to thank Jean-Philippe Anker for making him aware of the main question addressed in this paper, as well as for writing the excellent notes \cite{AnkerDunklNotes}, which have significantly lowered the barriers to entry for researchers interested in Dunkl theory. The work of C.M. is partially supported by the National Science and Technology Council of Taiwan under grant number 113WIA0110762.  The work of S.S. is partially supported by the Simons Foundation under grant number 00006698.

\bibliography{refs}
\bibliographystyle{amsplain}

\end{document}